\documentclass{amsart}
\usepackage{amssymb}

\hoffset=-0.3in \textheight=8in \theoremstyle{plain}

\newtheorem{theorem}{Theorem}

\newtheorem{lemma}{Lemma}
\newtheorem{proposition}{Proposition}

\theoremstyle{definition}
\theoremstyle{definition} \theoremstyle{remark}
\theoremstyle{remark}

\newcommand{\vs}{\vskip8pt}

\newcommand{\be}{\begin{enumerate}}
\newcommand{\ee}{\end{enumerate}}

\newcommand{\ZZ}{\mathbb{Z}}

\usepackage{graphicx}

\begin{document}

\title{Spread construction for (36,15,6) Hadamard difference sets}

\author[K. Smith] {Ken Smith}
\address [Ken Smith] {Department of Mathematics, Sam Houston State University}
\email {kenwsmith@shsu.edu}
\author[J. Webster] {Jordan Webster}
\address [Jordan Webster] {Department of
Mathematics, Mid Michigan Community College} \email{jdwebster@midmich.edu}

\begin{abstract}
There are exactly 35 inequivalent $(36,15,6)$ difference sets in nine groups.
Eight of the nine groups have a normal Sylow 3-subgroup.  
We give a straightforward ``spread" construction which explains the 32 inequivalent difference sets in these eight groups.
An interesting variation on this construction provides the three difference sets in the ninth group.
\end{abstract}

\maketitle

\section{Introduction}

\normalsize 
A $(v,k,\lambda)$ difference set is a subset $D$ of size $k$ in a group $G$ of order $v$ with the property that for every nonidentity $g\in G$, there are exactly $\lambda$ ordered pairs $(x,y) \in D \times D$ such that $xy^{-1}=g$. Difference sets with parameters $(4m^{2}, 2m^{2}\pm m, m^{2} \pm m)$ are often called ``Hadamard'' difference sets, ``Menon'' difference sets, ``Menon-Hadamard'' difference sets, or ``H-sets''. (For more details on difference sets in general, consult \cite{Lander} or chapter 11 of \cite{Hall}.)

Hadamard difference sets exist in many groups of order $4m^2$ including nonabelian groups.
 Kibler  (\cite{Kibler}) found that nine of the fourteen groups of order 36 contain $(36, 15, 6)$ difference sets. 
 He also found a total of 34 inequivalent difference sets in these nine groups. 
 Smith, in an exhaustive computer search (using {\em GAP}, \cite{GAP}) found another difference set which is inequivalent to the previous 34.
 The 35 difference sets are listed in \cite{CRC}. 
 In this paper we find a general construction that explains all the difference sets found by Kibler and Smith. 

The philosophy of this paper is motivated by the elegance (and ideas) of  \cite{Dillon}.

\vs
Throughout this paper we assume that $G$ is a finite (multiplicative) group. 
We work in the group ring $\ZZ[G]$,
the set of formal sums $\sum_{g\in G}{a_{g}g}$ where $a_{g}\in \ZZ$.
We identify a $D\subseteq G$ with the element $D=\sum_{g\in D}{g} $ in $\ZZ[G]$. 
Notice that this is a slight abuse of notation, but it will be clear whether we are working in the group or the group ring. 
One specific example which is often used in this paper is the notation $<a>$ which symbolizes the sum $\sum_{k=1}^{n}{a^{k}}$ in $\ZZ[G]$ where $n$ is the order of $a$. 
Also, if $X=\sum_{g\in G}{a_{g}g}$ we will denote $X^{(-1)}=\sum_{g\in G}{a_{g}g^{-1}}$. Thus, the $^{(-1)}$ operation ``switches'' the coefficients of $g$ and $g^{-1}$ for every $g\in G$. 
With this notation, an equivalent definition of a difference set using the group ring structure is an element $D$ of the ring with the property that 
\begin{equation} DD^{(-1)}= (k-\lambda)1_{G} + \lambda G \label{basic} \end{equation}
where $1_{G}$ is the identity of the group. 
(We will often omit writing the identity and assume that a scalar $c$ in the group ring will be thought of as $c\cdot 1_{G}$.)

\vs
When working with Hadamard parameters, it is convenient to use the Hadamard transform of the difference set rather than the difference set. 
If $D$ is a difference set (written as group ring element), then the Hadamard transform of $D$ is $\widehat{D}=G-2D$. 
If we have the Hadamard parameters $(4m^{2}, 2m^{2}\pm m, m^{2} \pm m)$, then from the basic equation \ref{basic}, above, the Hadamard transform has the property that 
\begin{equation}  \widehat{D}\widehat{D}^{(-1)}= 4m^{2}. \label{basic_hadamard}\end{equation}
The group ring element $\widehat{D}$ has coefficients $1$ and $-1$ in the place of the coefficients $0$ and $1$ in $D$. 
We may recover $D$ from $\widehat{D}$ via $D=\frac12(G-\widehat{D})$ and so finding the Hadamard transform of a difference set is equivalent to finding the difference set itself. 
In particular, if we can obtain a group ring element $S$ such that $S$ has coefficients of ones and negative ones and such that $S$ satisfies $SS^{(-1)}=4m^{2}$, then $S$ is the Hadamard transform of a difference set.

\section{Product constructions}

There are a number of product constructions for Hadamard difference sets.
Most depend on finding a single difference set or multiple difference sets in certain subgroups of a group and using some multiplication in the group ring to construct a difference set in the group.
 One such construction for Hadamard difference sets was given by Menon in \cite{Menon}.

\begin{theorem}[Menon]
If $G$, $H$, and $K$ are groups such that $G=H\times K$ and $D_{K}$ is a difference set in $K$ and $D_{H}$ is a difference set in $H$, then
$$\widehat{D}=\widehat{D_{H}}\cdot \widehat{D_{K}}$$
is the Hadamard transform of a difference set.
\end{theorem}


Dillon (\cite{Dillon}, remark on p. 13) generalized the product construction as follows:
\begin{theorem}[Dillon]
Suppose $G$ is a group with subgroups $H$ and $K$ such that $G=HK$, $H \cap K = \{1\}.$
Suppose furthermore that $D_{K}$ is a Hadamard difference set in $K$ and $D_{H}$ is a Hadamard difference set in $H.$
Then
$$\widehat{D}=\widehat{D_{H}}\cdot \widehat{D_{K}}$$
is the Hadamard transform of a difference set in $G$.
\end{theorem}

\vs
 
%
%
%
%
%

\section{The Spread Construction}
 
The generalized product construction will not be very useful in groups of order 36 since no group of order 36 has a proper subgroup with a Hadamard difference set, but we may use some of the general ideas in the construction to build something we call a \emph{spread construction}. 

A spread in a group $G$ is a collection  
of $q+1$ subgroups of $G$, each of order $q$, such that the intersection of any two subgroups is $\{1\}.$  
McFarland (\cite{McFarland}) uses spreads to construct an infinite family of difference sets in abelian groups of order $(q+2)q^2.$
Dillon  (\cite{Dillon}, pages 10-12) generalizes this to nonabelian groups, using spreads on one hand to construct Hadamard difference sets in groups of order $2^{2s}$ and, on the other hand, McFarland difference sets in nonabelian groups (of nonsquare order.)

Suppose a group $G$ of order $(q+2)q^2$ has a subgroup $H$ with a spread $\{L_0,L_1,...,L_q\}$ of $q+1$ subgroups of $H$, each of order $q$.  
Suppose also that e there is a left transversal $h_0, ..., h_{q-1}, h_{q}$ of $H$ in $G$ such that the map $L_i \mapsto h_iL_ih_i^{-1}$ is a permutation of the set 
$\{L_0,L_1,...,L_q\}$. 
Then a McFarland difference set  (with parameters $(q+2)q^2, (q+1)q, q)$) can be constructed in $G$
 (\cite{Dillon}, theorem on p. 12.) 
We seek a similar description of the $(36,15,6)$ Hadamard difference sets.


We assume for the remainder of the paper that $G$ is a group of order 36 with subgroup $H=<a,b : a^{3}=b^{3}=1 , ab=ba> \cong C_{3}\times C_{3}$ and that $L_0, L_1, L_2, L_3$ represent the four nontrivial proper subgroups of $H.$ 
These four subgroups form a spread.
The nine elements of $H$ form the points of an affine plane of order three. 
The 12 lines of the affine plane are the cosets of the four subgroups $L_0, L_1, L_2, L_3.$
\vs

We show that a modification of Dillon's construction from \cite{Dillon} describes {\em all} 32 difference sets arising in the eight groups in which the Sylow 3-subgroup $H$ is normal. 
 In the one group in which $H$ is not normal, the spread construction fails and in that case we provide an interesting alternative construction that {\em relies} on the fact that  $H$ in {\em not} normal.
 
\vs 

In the group ring $\ZZ[H]$ the subgroups $L_i$ obey the following identities.
\begin{enumerate}
\item[(0)] $L_i^{(-1)}=L_i.$
\item $L_iH=3H$
\item $L_i^2= 3L_i.$  If $i \ne j$ then $L_iL_j = H$.  
\item $\sum_{i=0}^3 L_i = 3+H.$
\end{enumerate} 

If we define $\widehat{L_{i}}:=H-2L_{i}$, the ``transform" of $L_i$ with respect to $H$ then we have the following:

\begin{enumerate}
\item[(0)] $\widehat{L_{i}}^{(-1)}=\widehat{L_{i}}$
\item $\widehat{L_i}H=3H$
\item $\widehat{L_{i}}^{2}=3H-6\widehat{L_{i}}$.  If $i \ne j$ then $\widehat{L_{i}}\widehat{L_{j}}=H.$
\item $\sum_{i=0}^{3}\widehat{L_{i}}=-6+2H$

\end{enumerate} 

Suppose that $x, y$ are in the normalizer of $H$ and that $\{1,x,y,xy\}$ is a left transversal of $H$ in $G$.
Consider the group ring element
 \begin{equation}\boxed{S=\widehat{L_{0}} -xh_1\widehat{L_{1}}-yh_2\widehat{L_{2}}-xyh_3\widehat{L_{3}}.}\end{equation}

With this notation
$$SS^{(-1)}
= \widehat{L_0}^2 + x\widehat{L_1}^2x^{-1} + y\widehat{L_2}^2y^{-1} +xy\widehat{L_3}^2x^{-1}y^{-1}$$
$$= 12H-6(\widehat{L_0} + x\widehat{L_1}x^{-1} + y\widehat{L_2}y^{-1} +xy\widehat{L_3}x^{-1}y^{-1})$$

Now 
$\widehat{L_{0}}+ \widehat{L_{1}} + \widehat{L_{2} }+ \widehat{L_{3}} = -6+2H$,  which is just what is required in the parentheses, above, to get the desire $36.$

Therefore
$SS^{(-1)} = 36$
if and only if 

$$\widehat{L_{0}}+ \widehat{L_{1}} + \widehat{L_{2} }+ \widehat{L_{3}} 
=
\widehat{L_0} + x\widehat{L_1}x^{-1} + y\widehat{L_2}y^{-1} +xy\widehat{L_3}x^{-1}y^{-1}$$

Equivalently 
$SS^{(-1)} = 36$ if and only if the sets 
 $\{L_0, L_1, L_2, L_3\}$ and  $\{L_0, xL_1x^{-1}, yL_2y^{-1}, xyL_3y^{-1}x^{-1}\}$ are equal.


%

\begin{theorem}[The Spread Construction]
Let $G$ be a group of order $36$ with normal subgroup $H.$ Let $1,x,y,xy$ be a transversal of $H$ in $G$ such that 
the sets $\{L_0, L_1, L_2, L_3\}$ and  $\{L_0, xL_1x^{-1}, yL_2y^{-1}, xyL_3y^{-1}x^{-1}\}$ are equal then 
 $$S = \widehat{L_0}-xh_1\widehat{L_1}-yh_2\widehat{L_2}-xyh_3\widehat{L_3}$$
 is the Hadamard transform of a difference set in $G$. 
\end{theorem}

This technique give us a way to search for difference sets in groups of order 36.
All 35 inequivalent difference sets with parameters $( 36, 15, 6)$ are in nine groups and are listed in \cite{CRC}. 
Below, we write out the first 32 difference sets as being equivalent to something that comes from the spread construction.
 We wish to write the difference set transforms in the form $S= \widehat{L_{0}} -x h_1\widehat{L_{1}} - yh_2\widehat{L_{2}}- xyh_3\widehat{L_{3}}$ where $h_i \in H$
In many of these groups, in order to clarify the correspondence between the spread construction and the notation of \cite{CRC},  \cite{Kibler}, we write $S$ as a translate $gD_ig'$ of the difference set given in\cite{CRC},  \cite{Kibler}.


\begin{center}
$G=<a,b,c: a^{3}=b^{3}=c^{4}=[a,b]=[a,c]=[b,c]=1 >$ (abelian)
\begin{tabular}{|c|c|c|c|c|c|c|}
\hline
{\bf $S$} \hspace{.25 in} & {\bf $L_{0}$} \hspace{.25 in} & {\bf $h_1L_{1}$} \hspace{.25 in} & {\bf $h_2L_{2}$} \hspace{.25 in} & {\bf $h_3L_{3}$} \hspace{.25 in} & {$x$} \hspace{.25 in} & {$y$} \hspace{.25 in} \\
\hline \hline
$b\widehat{D_{1}}$ & $<a>$   & $<b>$  & $b<ab>$  & $b<ab^{2}>$ & $c$ & $c^{2}$ \\
$b\widehat{D_{2}}$ & $<a>$   & $<b>$  & $b<ab>$  & $ab<ab^{2}>$ & $c$ & $c^{2}$ \\
$b\widehat{D_{3}}$ & $<a>$   & $<b>$  & $b<ab>$  & $<ab^{2}>$ & $c$ & $c^{2}$ \\
$b\widehat{D_{4}}$ & $<a>$   & $<b>$  & $<ab>$  & $<ab^{2}>$ & $c$ & $c^{2}$ \\
\hline
\end{tabular}
\end{center}


\vs
\begin{center}
$G=<a,b,c: a^{3}=b^{3}=c^{4}=[a,b]=[a,c]=b^{2}cb^{-1}c^{-1}=1 >$
\begin{tabular}{|c|c|c|c|c|c|c|}
\hline

{\bf $S$} \hspace{.25 in} & {\bf $L_{0}$} \hspace{.25 in} & {\bf $h_1L_{1}$} \hspace{.25 in} & {\bf $h_2L_{2}$} \hspace{.25 in} & {\bf $h_3L_{3}$} \hspace{.25 in} & {$x$} \hspace{.25 in} & {$y$} \hspace{.25 in} \\

\hline \hline
$\widehat{D_{5}}b$ & $<a>$   & $b<ab^{2}>$  & $<b>$  & $b<ab>$ & $c$ & $c^{2}$ \\
$\widehat{D_{6}}b$ & $<a>$   & $b<ab^{2}>$  & $<b>$  & $<ab>$ & $c$ & $c^{2}$ \\
$a\widehat{D_{7}}$ & $<b>$   & $a<ab^{2}>$  & $<a>$  & $a<ab>$ & $c$ & $c^{2}$ \\
$a\widehat{D_{8}}$ & $<b>$   & $a<ab^{2}>$  & $<a>$  & $a^2<ab>$ & $c$ & $c^{2}$ \\
$a^2\widehat{D_{9}}b^2$ & $<a^{2}b>$   & $b^2<a>$  & $<ab>$  & $a^2<b>$ & $c$ & $c^{2}$ \\
$a^2\widehat{D_{10}}b^2$ & $<a^{2}b>$   & $b^2<a>$  & $<ab>$  & $<b>$ & $c$ & $c^{2}$ \\
\hline
\end{tabular}
\end{center}

%
%

\vs
\begin{center}
$G=<a,b,c: a^{3}=b^{3}=c^{4}=[a,b]=a^{2}ca^{-1}c^{-1}=b^{2}cb^{-1}c^{-1}=1 >$
\begin{tabular}{|c|c|c|c|c|c|c|}
\hline

{\bf $S$} \hspace{.25 in} & {\bf $L_{0}$} \hspace{.25 in} & {\bf $h_1L_{1}$} \hspace{.25 in} & {\bf $h_2L_{2}$} \hspace{.25 in} & {\bf $h_3L_{3}$} \hspace{.25 in} & {$x$} \hspace{.25 in} & {$y$} \hspace{.25 in} \\

\hline \hline
$\widehat{D_{11}}b$ & $<a>$   & $<ab>$  & $<b>$  & $b<ab^{2}>$ & $c$ & $c^{2}$ \\
\hline
\end{tabular}
\end{center}

%
%

\vs
\begin{center}
$G=<a,b,c: a^{3}=b^{3}=c^{4}=bca^{-1}c^{-1}=a^{2}cb^{-1}c^{-1}=1 >$ 
\begin{tabular}{|c|c|c|c|c|c|c|}
\hline

{\bf $S$} \hspace{.25 in} & {\bf $L_{0}$} \hspace{.25 in} & {\bf $h_1L_{1}$} \hspace{.25 in} & {\bf $h_2L_{2}$} \hspace{.25 in} & {\bf $h_3L_{3}$} \hspace{.25 in} & {$x$} \hspace{.25 in} & {$y$} \hspace{.25 in} \\

\hline \hline
$\widehat{D_{12}}b$ & $<a>$   & $ab<ab^{2}>$  & $<b>$  & $b<ab>$ & $c$ & $c^{2}$ \\
$\widehat{D_{13}}b$ & $<a>$   & $ab<ab^{2}>$  & $<b>$  & $b^2<ab>$ & $c$ & $c^{2}$ \\
$\widehat{D_{14}}b$ & $<a>$   & $ab<ab^{2}>$  & $<b>$  & $<ab>$ & $c$ & $c^{2}$ \\
$\widehat{D_{15}}b$ & $<a>$   & $ab<ab^{2}>$  & $a^{2}<b>$  & $b<ab>$ & $c$ & $c^{2}$ \\
$\widehat{D_{16}}b$ & $<a>$   & $b^2<ab>$  & $<b>$  & $b<ab^{2}>$ & $c$ & $c^{2}$ \\
\hline
\end{tabular}
\end{center}

%
%

\vs
\begin{center}
$G=<a,b,c: a^{3}=b^{3}=c^{2}=d^{2}=[a,b]=[c,d]=[a,c]=[a,d]=[b,c]=[b,d]=1 >$ (abelian)
\begin{tabular}{|c|c|c|c|c|c|c|}
\hline

{\bf $S$} \hspace{.25 in} & {\bf $L_{0}$} \hspace{.25 in} & {\bf $h_1L_{1}$} \hspace{.25 in} & {\bf $h_2L_{2}$} \hspace{.25 in} & {\bf $h_3L_{3}$} \hspace{.25 in} & {$x$} \hspace{.25 in} & {$y$} \hspace{.25 in} \\

\hline \hline
$a^2c\widehat{D_{17}}$ & $<b>$ & $<a>$ & $<ab^{2}>$  & $<ab>$ & $c$ & $d$ \\
$a^2c\widehat{D_{18}}$ & $<b>$ & $<a>$  & $a^2<ab>$  & $a<ab^{2}>$ & $c$ & $d$ \\
$a^2c\widehat{D_{19}}$ & $<ab^{2}>$ & $<a>$ & $a^2<ab>$  & $a^2<b>$ & $c$ & $d$ \\
\hline
\end{tabular}
\end{center}

%
%

\vs

\begin{center}
$G=<a,b,c: a^{3}=b^{3}=c^{2}=d^{2}=[a,b]=[c,d]=[a,d]=[b,d]=1, a^{2}c=ca, b^{2}c=cb >$
\begin{tabular}{|c|c|c|c|c|c|c|}
\hline

{\bf $S$} \hspace{.25 in} & {\bf $L_{0}$} \hspace{.25 in} & {\bf $h_1L_{1}$} \hspace{.25 in} & {\bf $h_2L_{2}$} \hspace{.25 in} & {\bf $h_3L_{3}$} \hspace{.25 in} & {$x$} \hspace{.25 in} & {$y$} \hspace{.25 in} \\

\hline \hline
$\widehat{D_{20}}b$ & $<a>$   & $<b>$  & $b<ab>$  & $b<ab^{2}>$ & $c$ & $d$ \\
\hline
\end{tabular}
\end{center}

%
%

\vs

\begin{center}
$G=<a,b,c: a^{3}=b^{3}=c^{2}=d^{2}=[a,b]=[c,d]=[a,d]=[b,c]=[b,d]=a^{2}ca^{-1}c^{-1}=1 >$ 
\begin{tabular}{|c|c|c|c|c|c|c|}
\hline

{\bf $S$} \hspace{.25 in} & {\bf $L_{0}$} \hspace{.25 in} & {\bf $h_1L_{1}$} \hspace{.25 in} & {\bf $h_2L_{2}$} \hspace{.25 in} & {\bf $h_3L_{3}$} \hspace{.25 in} & {$x$} \hspace{.25 in} & {$y$} \hspace{.25 in} \\

\hline \hline
$b\widehat{D_{21}}$ & $<a>$   & $b<ab^{2}>$  & $<b>$  & $b<ab>$ & $c$ & $d$ \\
$b\widehat{D_{22}}$ & $<a>$   & $b<ab^{2}>$  & $<b>$  & $a<ab>$ & $c$ & $d$ \\
$\widehat{D_{23}}a$ & $<b>$   & $a<ab^{2}>$  & $<a>$  & $a<ab>$ & $c$ & $d$ \\
$\widehat{D_{24}}a$ & $<b>$   & $a<ab^{2}>$  & $<a>$  & $<ab>$ & $c$ & $d$ \\
$\widehat{D_{25}}a$ & $<ab^{2}>$   & $<a>$  & $a<ab>$  & $a<b>$ & $c$ & $d$ \\
$\widehat{D_{26}}a$ & $<ab^{2}>$   & $<a>$  & $<ab>$  & $a<b>$ & $c$ & $d$ \\
\hline
\end{tabular}
\end{center}

%
%

\vs

\begin{center}
$G=<a,b,c: a^{3}=b^{3}=c^{2}=d^{2}=[a,b]=[c,d]=[a,d]=[b,c]=1, b^{2}d=db, a^{2}c=ca >$
\begin{tabular}{|c|c|c|c|c|c|c|}
\hline

{\bf $S$} \hspace{.25 in} & {\bf $L_{0}$} \hspace{.25 in} & {\bf $h_1L_{1}$} \hspace{.25 in} & {\bf $h_2L_{2}$} \hspace{.25 in} & {\bf $h_3L_{3}$} \hspace{.25 in} & {$x$} \hspace{.25 in} & {$y$} \hspace{.25 in} \\

\hline \hline
$\widehat{D_{27}}b$ & $<a>$   & $b<ab^{2}>$  & $b<ab>$  & $<b>$ & $c$ & $d$ \\
$\widehat{D_{28}}b$ & $<a>$   & $b<ab^{2}>$  & $b<ab>$  & $a<b>$ & $c$ & $d$ \\
$\widehat{D_{29}}a$ & $<ab^{2}>$   & $<a>$  & $a<b>$  & $a<ab>$ & $c$ & $d$ \\
$\widehat{D_{30}}a$ & $<ab^{2}>$   & $<a>$  & $<b>$  & $a<ab>$ & $c$ & $d$ \\
$\widehat{D_{31}}a$ & $<ab^{2}>$   & $a<b>$  & $<a>$  & $a<ab>$ & $c$ & $d$ \\
$\widehat{D_{32}}a$ & $<ab^{2}>$   & $a<b>$  & $<a>$  & $<ab>$ & $c$ & $d$ \\
\hline
\end{tabular}
\end{center}

%
%

\vs
This shows that all 32 difference sets in the first eight groups having difference sets in \cite{CRC} can be explained by the spread construction. 



\section{Using Relative Difference Sets}

The last group of order 36 which contains difference sets is the group listed in Kibler's tables (\cite{Kibler}) as $G=<a,b,c: a^{3}=b^{3}=c^{2}=d^{2}=[a,b]=[c,d]=[b,c]=[b,d]=1, da=ac, cda=ad >$. 
This group, {\em SmallGroup(36,11)} in the catalogue of small groups in {\em GAP} \cite{GAP}, is isomorphic to $A_4 \times C_3.$  
Careful analysis of the situation in this group reveals that {\em no} version of the lines $L_0, L_1, L_2$ and $L_3$ will create difference sets in this final group.
Fortunately there is a nonabelian alternative, replacing a spread of lines with a configuration of relative difference sets.





\vs

A $(m,n,k,\lambda )$ relative difference set $R$  (size $k$) in a group $G$ (of order $mn$) with respect to a forbidden subgroup $N$ (of order $n$) is a set $R$ with the property that 

\begin{equation}
RR^{(-1)}= k + \lambda(G-N)
\end{equation}
 
 Relative difference sets are introduced in \cite{Butson}; recommended references on relative difference sets are  \cite{Jungnickel} or section 3.1.4 of \cite{Schmidt}.
 


In $H=C_{3}\times C_{3}=<a,b>$ the subsets $\{1, a, a^{2}b^{2} \}$, $\{1,a^{2}, ab^{2}\}$, and $\{1,ab, a^{2}b\}$ are all $(3,3,3,1)$ relative difference sets with respect to the forbidden subgroup $<b>$. 
A $(3,3,3,1)$ relative difference set in $H$ has the property that if $\chi$ is principle on $<b>$ then $\chi(R) = 0;$ if $\chi$ is not principle on $<b>$ then $\chi(R)\overline{\chi(R)}= \chi(RR^{(-1)})=3.$

Let $R := 1+a+a^2b, R_1:= 1+a+a^2b^2$ and $R_2=R_1^{(-1)} = 1+ab+a^2$.
These are relative $(3,3,3,1)$ difference sets with respect to $N := <b>.$
Again, we let $L_0 = <a>$ and $\widehat{L_0}=H-2L_0.$
Similarly $\widehat{R_1} := H-2R_1$ and $\widehat{R_2} := H-2R_2$.

In the last group of order 36, $G=A_4 \times C_3$, there are three inequivalent difference sets.  In \cite{CRC}, these are listed as $D_{33}, D_{34}$ and $D_{35}$.  
Although these difference sets cannot be written in  terms of a spread, they may be written in an analogous way using relative difference sets.

We take advantage of the structure of this single group  $G=A_4 \times C_3$.
  We may write $A_4$ in terms of an element $a$ of order three and an involution $c$. 
   (One may specifically choose $a = (123)$ and $c=(12)(34)$. 
   Let $b$ be an element of order 3 which is in the center of $G$ and set $H:= <a,b>.$

  
  Since $<a>$ is not normal in $A_4$, then $H$ is not normal in $G$.
  
 \begin{lemma}
Suppose $R_0 = \{r_1, r_2, r_3\}$ and $R_1$ are relative difference sets in $H$ with respect to $L_0=<b>.$  
Suppose, furthermore that there is an element $g$ such that 
$(1+R_0g)H=G,$
(so $\{1, r_1g, r_2g, r_3g\}$ is a {\em left} transversal of $H$ in $G$.)
Then
 \begin{equation}\boxed{S=\widehat{L_{0}} -R_0g\widehat{R_1}}\end{equation}
 is the transform of a $(36,15,6)$ Hadamard difference set.
 \end{lemma}

Observe that since $\{1, r_1g, r_2g, r_3g\}$ is a {\em left} transversal of $H$ in $G$ and since $R_1 \subseteq H$, the coefficients of $S$ are $\pm1$.  (Note also that although $\{1, r_1g, r_2g, r_3g\}$ is a left transversal of $H$ in $G$ it can{\em not} be a right transversal; indeed all the members of the set are in the right coset $Hg$!  Thus this construction makes sense {\em only} if $H$ is {\em not} normal in $G$.)

We show that $S$ obeys the equation 
 $SS^{(-1)}=36$ by showing that under all irreducible representations of $G$,  $SS^{(-1)}$ is mapped to 36.
For example, under the trivial representation, both $\widehat{L_0}$ and $\widehat{R_1}$ are mapped to 3.
Thus  $S$ is mapped to $3-(3)(1)(3)=-6$ so $SS^{(-1)}$ is mapped to 36.

$A_4$ has three linear representations and a single degree three irreducible representation.  
Let us write each element of $G$ uniquely in the form $\sigma b^j$ where $\sigma \in A_4.$

Set $\omega := e^{2\pi/3},$ a cube root of unity.  
The irreducible representations of $<b>$ merely send $b$ to one of three cube roots of unity, $1, \omega,$ or $ \bar{\omega}.$

Given a  irreducible representation $\phi$ of $A_4$ and  irreducible representation $\chi$ of $<b>$, the function which maps 
$\sigma b^j$ to $\phi(\sigma) \cdot \chi(b)^j$ is an irreducible representation of $G$.  
Indeed, all irreducible representation of $G$ have this form.

\vs
\subsection{The nonlinear irreducible representations}

Suppose that $\phi$ is the irreducible degree three representation of $A_4$.  
Then $\phi$ is equivalent to a representation which sends $a$ to 
$\begin{pmatrix} 0 & 1 & 0 \\ 0 & 0 & 1 \\ 1 & 0 & 0 \end{pmatrix}$
and so $\phi(<a>)= J_3$, the $3\times 3$ matrix of all ones.
If the representation $\phi$ also maps $b$ to the identity matrix then 
$\phi(H) = 3J_3 \text{ and } \phi(R_0) = \phi(\widehat{R_1}) = J_3$.
Since $\phi(<b>) = 3 I_3$ then 
$\phi(\widehat{L_0}) =    3J_3 - 6I_3.$

An element of $A_4$ not in $H$ is mapped to a signed matrix with exactly two $-1$s and one $1$.
(For example, $c$ might be mapped to 
$\begin{pmatrix} -1 & 0 & 0 \\ 0 & -1 & 0 \\ 0 & 0 & 1 \end{pmatrix}$
and $ac$ to
$\begin{pmatrix}  0 & 0 & 1 \\ -1 & 0 & 0 \\ 0 & -1 & 0 \\ \end{pmatrix}$.)
Since $g \not\in H$ then
$\phi(R_0g\widehat{R_1}) = J_3,$
and so 
$$\phi(S) =  4J_3 - 6I_3 =   \begin{pmatrix} -2 & 4 & 4 \\ 4 & -2 & 4 \\ 4 & 4 & -2 \\ \end{pmatrix}.$$
One may then verify that $\phi(S)\overline{\phi(S)}^t = 36 I_3.$

\vs
If the irreducible degree 3 representation maps $b$ to $\omega I_3$ (or $\bar{\omega} I_3$) then  $\phi(H) =\phi(<b>) = 0.$
Since $R_0R_0^{(-1)} = 3 + H-<b>$ then $\phi(R_0)\overline{\phi(R_0)}^t = 3 I_3.$
Thus $\phi(S) = \phi(R_0)\phi(g)\phi(\widehat{R_1})= -2\phi(R_0)\phi(g)\phi(R_1)$
and so 
$$\phi(SS^{(-1)})= \phi(S)\phi(S^{(-1)}) = 4\phi(R_0)\phi(g)\phi(R_1R_1^{(-1)})\phi(g^{-1})\phi(R_0^{(-1)}).$$
Replacing $\phi(R_1R_1^{(-1)})$ by $3I_3$ and observing that $\phi(g)3I_3\phi(g^{-1}) = 3I_3,$ this simplifies to 
$$ = 4\phi(R_0)(3I_3)\phi(R_0^{(-1)})= 12\phi(R_0)\phi(R_0^{(-1)})
= 12(3I_3)=36I_3.$$
Again, we have that $\phi(S)\overline{\phi(S)}^t = 36 I_3.$

\vs
\subsection{The linear representations of $G$}

The nontrivial linear representations of $G$ send the elements of order two (such as  $c, d$) to 1 and act as linear characters of $H$. 
Since $R_0$ and $R_1$ are relative difference sets then if
$b$ is mapped to 1, then $R_0$ and $R_1$ (and $\widehat{R_1}$) are mapped to zero.
$\widehat{L_0} = H-2<b>$ is mapped to 6 and so $S$ is mapped to 6; clearly $SS^{(-1)}$ is then mapped to 36.

If $b$ is mapped to $\omega$ or $\bar{\omega}$ then $<b>$ and $\widehat{L_0}$ are mapped to zero while $R_0$ and $R_1$ are mapped to complex numbers of length $\sqrt{3}$.
The element $\widehat{R_1}$ is mapped to $-2$ times the image of $R_1.$

Therefore $S$ is mapped to the image of $-R_0g\widehat{R_1}$ which is 2 times the image of $R_0R_1g$ and so is mapped to a complex number of length 6.  Again, 
$SS^{(-1)}$ is mapped to 36.

In all irreducible representations of $G$, $SS^{(-1)}$ is mapped to 36 and so $SS^{(-1)}=36.$

So, as long as  $\{1,r_1, r_2, r_3\}$ is a left transversal for $H$ in $G$, the group ring element $S$ is the transform of a Hadamard difference set.

\begin{proposition}
$D_{33}$, $D_{34}$, and $D_{35}$ are each equivalent to a difference set  which may be written as $\widehat{L_0} - (g_1+ag_1+a^2bg_1)R_1$ where $1+a+a^2b$ and $R_1$ are relative difference sets in $<b,a>$ with forbidden subgroup $<b>$.
\end{proposition}

\begin{proof}
 The next table shows what each of these are.

\begin{center}
$G=<a,b,c: a^{3}=b^{3}=c^{2}=d^{2}=[a,b]=[c,d]=[b,c]=[b,d]=1, da=ac, cda=ad >$
\begin{tabular}{|c|c|c|c|c|c|c|}
\hline
{\bf $S$} \hspace{.25 in} & {\bf $L_{0}$} \hspace{.25 in} & {\bf $R_{1}$} \hspace{.25 in} & {\bf $R_{2}$} \hspace{.25 in} & {\bf $R_{3}$} \hspace{.25 in} & {$x$} \hspace{.25 in} & {$y$} \hspace{.25 in} \\
\hline \hline
$cD_{33}$ & $a^{2}<b>$   & $1+a+a^{2}b^{2}$  & $1+ab+a^{2}b$  & $1+a^{2}+ab^{2}$ & $c$ & $d$ \\
$cdD_{34}^{c}$ & $<b>$   & $1+ab^{2}+a^{2}b^{2}$  & $1+a^{2}+ab$  & $1+a+a^{2}b$ & $c$ & $d$ \\
$D_{35}^{cd}$ & $a^{2}<b>$   & $1+a^{2}+ab$  & $1+a+a^{2}b$  & $b^{2}+ab+a^{2}b$ & $c$ & $d$ \\
\hline
\end{tabular}
\end{center}

\end{proof}

\section{Conclusions}

We have shown that in the eight groups of order 36 with difference sets then all difference sets can be explained by the spread construction and the difference sets in the last group can be explained by a spread construction using relative difference sets. A full list of difference sets with parameters $(144,66,30 )$ is not known, but it would be interesting to see if all difference sets have a spread construction involved. The authors know that in groups of order 144, of those with a noncyclic Sylow-3 subgroup, using various nonexistence results and a version of the spread construction, one can conclude existence or nonexistence of Hadamard difference sets in all but one group of order 144. The group action in the unknown group of order 144 is similar to the group action in the group of order 36 whose difference sets do not come from the spread construction. It would be interesting to see whether or not the relative difference set construction could be altered to work in this group as it works in the last group of order 36.


\begin{thebibliography}{9}

\bibitem{Bhatt}
C. Bhattacharya and K. Smith, Factoring (16,6,2) Hadamard difference sets, \emph{Electronic Journal of Combinatorics} \textbf{13} (2006), \#R00.
\vspace{1 mm}

\bibitem{CRC}
K.W. Smith, Difference sets: nonabelian, in the CRC Handbook of Combinatorial Designs, Colburn and Dinitz, eds., CRC Press, 1996, 308-312.

\bibitem{Dillon}
 J.F. Dillon, Variations on a scheme of McFarland for noncyclic difference sets. J. Combin. Theory Ser. A 40 (1985), 9-21.

\bibitem{GAP}
The GAP~Group, \emph{GAP -- Groups, Algorithms, and Programming, 
  Version 4.4.12};   2008,
  \verb+(http://www.gap-system.org)+.

\bibitem{Hall}
Marshall Hall, Combinatorial Theory, John Wiley and Sons, 1998.

\bibitem{Kibler}
R.E. Kibler, A summary of noncyclic difference sets, $k< 20$, \emph{J. Combinatorial Theory} Ser. A 25, no. 1 (1978), 62-67.
\vspace{1 mm}

\bibitem{Lander}
E.S. Lander, ``Symmetric Designs: An Algebraic Approah,'' London Mathematical Society Lecture Note Seeries, Vol. 74, Cambridge University Press, Cambridge, 1983.
\vspace{1 mm}

\bibitem{McFarland}
R. McFarland, A family of difference sets in noncyclic groups, J. Combin. Theory Ser. A 15 (1973), 1-10.

\bibitem{Menon}
P.K. Menon, On difference sets whose parameters satisfy a certain relation, \emph{Proceedings of the American Mathematical Society} 13 (1962), 739-745.
\vspace{1 mm}

\bibitem{Schmidt}
Bernhard Schmidt, Characters and Cyclotomic Fields in Finite Geometry, Springer-Verlag, 2002.

\bibitem{Butson} A. T. Butson, Relations among generalized Hadamard matrices, relative difference sets, and maximal length linear recurring sequences, Canad. J. Math. 15, 42-48 (1963).

\bibitem{Jungnickel}
 D. Jungnickel, 'On automorphism groups of divisible designs', Canad. J. Math. 34 (1982) 257-297.
\vfill

\end{thebibliography}
\end{document}